\documentclass[11pt]{amsart}
 \usepackage[dvips]{epsfig}
\usepackage{latexsym}
\usepackage{amsfonts,amsmath,amssymb}
\newtheorem{theorem}{Theorem}[section]
\newtheorem{proposition}[theorem]{Proposition}
\newtheorem{lemma}[theorem]{Lemma}

\newtheorem{corollary}[theorem]{Corollary}
\newtheorem{example}[theorem]{Example}

\author[M. B\'ona]{Mikl\'os  B\'ona}
\title[A new record for 1324-avoiding permutations]{A new record for 1324-avoiding permutations}
\address{\rm M. B\'ona, Department of Mathematics, 
University of Florida,
358 Little Hall, 
PO Box 118105, 
Gainesville, FL 32611--8105 (USA)
}

\date{\today}

\begin{document}

\begin{abstract} Refining an existing counting argument, we
 provide an improved upper bound for the number of 1324-avoiding permutations of a given length.
\end{abstract}



\maketitle

\begin{center} {\em  To Richard Stanley, on the occasion of his seventieth birthday.} \end{center}

\section{Introduction}
\subsection{Definitions and Open Questions}
 The theory of pattern avoiding permutations has
seen tremendous progress during the last two decades. The key definition is
the following. Let $k\leq n$, let $p=p_1p_2\cdots p_n$ be a permutation of
 length $n$, and
let $q=q_1q_2\cdots q_k$ be a permutation of length $k$. We say that $p$
 avoids $q$ if there 
are no $k$ indices $i_1<i_2<\cdots <i_k$ so that for all $a$ and $b$, 
the inequality $p_{i_a}<p_{i_b}$ holds if and only if the inequality $q_a<q_b$
holds. For instance, $p=2537164$ avoids $q=1234$ because $p$ does not contain
an increasing subsequence of length four. See \cite{combperm} for an overview
of the main results on pattern avoiding permutations.

The shortest pattern for which even some of the most basic questions are open
is $q=1324$, a pattern that has been studied for at least 19 years.
 For instance,
 there is no known exact formula for the number 
$S_n(1324)$ of permutations of length $n$ (or, in what follows,
 $n$-permutations)
 avoiding 1324. Even the 
value of $L(1324)=\lim_{n\rightarrow \infty} \sqrt[n]{S_n(1324)}$ is unknown, though
the limit is known to exist.  Indeed, a spectacular result of Adam Marcus and 
G\'abor Tardos \cite{marcus} shows that for all patterns $q$, there exists a constant
$c_q$ so that $S_n(q)\leq c_q^n$ for all $n$, and a short argument \cite{arratia} then shows
that this implies the existence of  $L(q)=\lim_{n\rightarrow \infty} \sqrt[n]{S_n(q)}$. It is also known that no
pattern of length four is easier to avoid than the pattern 1324, that is,
for any pattern $q$ of length 4, the inequality $S_n(q)\leq S_n(1324)$ holds.
The inequality is sharp unless $q=4231$. See Chapter 4 of \cite{combperm}
for a treatment of the series of results leading to these inequalities.  

Recently, there has been some progress in the very challenging problem of determining $L(1324)$. First, 
in 2011, Anders Claesson, Vit Jelinek and Einar Steingr\'{\i}msson has proved that $L(1324)\leq 16$ \cite{claesson}.
A year later, the present author improved that upper bound \cite{bona-upper} showing that $L(1324)\leq 7+4\sqrt{3}<13.93$. As far as lower bounds go, in 2005 Albert and al. proved that $L(1324)\geq 9.42$, which has
recently been improved to $L(1324)\geq 9.81$ by David Bevan \cite{bevan}.

In the present paper, we further refine the counting argument of \cite{bona-upper}, leading to the improved upper
bound $L(1324)\leq 13.73718$. While this improvement is numerically not huge, the way in which it is obtained
is interesting as we are able to exploit some dependencies
between two words related to 1324-avoiding permutations that have previously been counted as though they were
independent. 

\section{Words over a finite alphabet}
\subsection{Coloring} \label{coloring}
The starting point of our proof is the following decomposition of 1324-avoiding
permutations, given in \cite{claesson} in a slightly different form, and then given in the present form
in \cite{bona-upper}. 

Let $p=p_1p_2\cdots p_n$ be a 1324-avoiding permutation, and let us color each 
entry of $p$ red or blue as we move from left to right, 
according the following rules. 
\begin{enumerate}
\item[(I)] If coloring $p_i$ red would create a 132-pattern with all red entries,
then color $p_i$ blue, and
\item[(II)] otherwise color $p_i$ red.  
\end{enumerate}

It is then proved in \cite{claesson} that the red entries form a 132-avoiding
permutation and the blue entries form a 213-avoiding permutation. 
The following fact will be useful in the next section. 

\begin{proposition} \label{blue} \cite{bona-upper}
 If an entry $p_i$ is larger than a blue entry on its left, then $p_i$ itself must be blue. 
\end{proposition}

If $p=p_1p_2\cdots p_n$ is a permutation, then we say that $p_i$ is a {\em left-to-right minimum} if $p_j>p_i$ for
all $j<i$. That is, $p_i$ is a left-to-right minimum if it is smaller than all entries on its left. Similarly, we say that
$p_k$ is a right-to-left maximum if $p_k>p_\ell$ for all $\ell >k$. In other words, $p_k$ is a right-to-left maximum if  it is larger than all entries on its right. It is easy to see that $p_1$ and 1 are always left-to-right minima, 
$p_n$ and $n$ are always right-to-left maxima, and both the sequence of left-to-right minima and the sequence of 
right-to-left maxima are decreasing.

In \cite{bona-upper}, this decomposition
of 1324-avoiding permutations  was refined as follows. 

Let us color each entry of the 1324-avoiding permutation 
$p=p_1p_2\cdots p_n$ red or blue as above. 
Furthermore, let us mark each entry of $p$ with one of the letters $A$,
 $B$, $C$, or
$D$ as follows.

\begin{enumerate}
\item Mark each red entry that is a left-to-right minimum in the partial 
permutation of red entries by $A$, 
\item mark each red entry that is not a left-to-right minimum in the partial
permutation of red entries by $B$, 
\item mark each blue entry that is not a right-to-left maximum in the partial
 permutation
of blue entries by $C$, and
\item mark each blue entry that is a right-to-left maximum in the partial
 permutation of
blue entries by $D$.
\end{enumerate}

Call entries marked by the letter $X$ entries of {\em type} $X$.
Let $w(p)$ be the $n$-letter word over the alphabet $\{A,B,C,D\}$ defined above.
In other words, the $i$th letter of $w(p)$ is the type of $p_i$ in $p$. 
Let $z(p)$ be the $n$-letter word over the alphabet $\{A,B,C,D\}$ whose $i$th
letter is the type of the entry of value $i$ in $p$. 

\begin{example} Let $p=3612745$. Then the subsequence of red entries of $p$
is $36127$, the subsequence of blue entries of $p$ is $45$, so 
$w(p)=ABABBCD$, while $z(p)=ABACDBB$. 
\end{example}

The following important fact is proved in \cite{bona-upper}. Let $\hbox{NO}(CB)_n$ be the set of all words of length 
$n$ over the alphabet $\{A,B,C,D\}$ that contain no $CB$-factors, that is, that do not contain a letter $C$ immediately
followed by a letter $B$. Let $\hbox{Av}(1324)_n$ denote the set of all 1324-avoiding permutations of length $n$.

\begin{theorem} \label{nocb}
 The map $f:Av(1324)_n \rightarrow NO(CB)_n \times NO(CB)_n$  given by $f(p)=(w(p),z(p))$ is
an injection. 
\end{theorem} 

Consequently, $S_n(1324)=|{Av}(1324)_n| \leq | NO(CB)_n|^2$, leading to the upper bound
\begin{equation} \label{oldone} S_n(1324)\leq (7+4\sqrt{3})^n .\end{equation}

\section{A connection between $w(p)$ and $z(p)$.}
\subsection{A small modification of the encoding}
One wasteful element of the method used to prove the upper bound (\ref{oldone}) that we sketched in the
last section is that it treats $w(p)$ and $z(p)$ as two independent words for enumeration purposes, which leads
to  the bound $|Av(1324)_n|\leq |NO(CB)_n|^2$. The full extent of the dependencies between 
$w(p)$ and $z(p)$ is difficult to describe, but in this paper, we describe it well enough to improve the 
upper bound for the exponential growth rate of $S_n(1324)$. 

First, we modify the letter encoding of 1324-avoiding permutations a little bit by adding the following last rule to 
the existing rules [(1)-(4)]. 

\begin{enumerate}
\item[(4')] If an entry is a right-to-left maximum in {\em all of $p$}, but not a left-to-right minimum in $p$, then 
color it blue and mark it $D$, regardless what earlier rules said. 
\end{enumerate}

In the rest of this paper, we will use rules (1)-(4) together with (4'). With a slight abuse of notation, for the sake
of brevity, we denote the pair of words encoding $p$ by these rules by $w(p)$ and $z(p)$, just as we did before rule
(4') was added. This will not create confusion, since the older, narrower set of rules will not be used. The advantage
of adding rule (4') is that it ensures that every letter $C$ is eventually followed by a letter $D$ so that the entry
corresponding to that $D$ is larger than the entry corresponding to that given $C$. 

Note that if an entry changes its letter code because of this last rule, then that entry was previously a $B$, and now
it is a $D$. It is straightforward to check that it is still true that the string of red entries forms a 132-avoiding permutation,
(since their string simply lost a few entries)
and the string of blue entries forms a 213-avoiding permutation. Indeed, if $u$ is a right-to-left minimum in $p$ that has
just become blue because of rule(4'), then $u$ could only play the role of 3 in a blue 213-pattern $yxu$. However, that would mean that before rule (4') was applied, the red entry $u$ was on the right of the smaller blue entries $x$ and $y$, 
contradicting Proposition \ref{blue}.  It is also clear that if $p\in Av(1324)_n$, then 
neither $w(p)$ nor $z(p)$ contains a $CB$-factor (since the set of letters $B$ shrank), and it is straightforward to 
show that the new map $p\rightarrow (w(p),z(p))$ is still injective.

\subsection{A connection between $w(p)$ and $z(p)$}
We are now ready to describe some connections between $w(p)$ and $z(p)$. 
A {\em segment} in a finite word $v$ over the alphabet $\{A,B,C,D\}$ is a subword of consecutive letters that
starts with an $A$ and ends immediately before the next letter $A$, or at the end of $v$. For instance, the
word $v=ABBDCACDB$ has two segments. 
The key observation is the following. 

\begin{lemma} \label{firstconn} Let $p$ be a 1324-avoiding permutation, and let $(w(p),z(p))$ be its image under
the coloring defined by rules (1)-(4) and (4'). 
Then for all $i$, the following holds.
If the $i$th letter $A$ from the right in $w(p)$ is in the middle of a $CAB$-factor, then the $i$th segment of
$z(p)$ from the left must contain a letter $B$.
\end{lemma}

\begin{proof}
Let us assume that the statement does not hold, that is, there is an $i$ that provides a counterexample.
Let $a$ be the $i$th smallest left-to-right minimum in $p$. That means that $a$ corresponds to the $i$th letter $A$
from the left in $z(p)$, and to the $i$th letter $A$ from the right in $w(p)$. Let us say that the letters corresponding to 
the entries $cab$ of $p$ form a $CAB$ factor in $w(p)$.
 Find the closest entry $d$ on the right of $b$ that corresponds to a 
letter $D$, and the closest entry $a'$ on the left of $c$ that corresponds to a letter $A$. Then the entries
$a'cbd$ will form a 1324-pattern in $p$ {\em unless} $b<a'$. 
However, that would mean that $a<b<a'$,  that is, there is a letter $B$ in the $i$th segment of $z(p)$ after all,
contradicting the assumption that $i$ is a counterexample.
\end{proof}

 The following simple proposition will be useful for us. 

\begin{proposition} \label{segments}
Let $s_n$ be the number of segments of length $n$ that do not contain a $CB$-factor. Let $S(x)=\sum_{n\geq 1}
s_nx^n$. Then the equality
\[S(x)=\frac{x}{x^2-3x+1}\] holds.
\end{proposition}

\begin{proof}
By definition, $s_0=0$ (since any segment must contain a letter $A$ at its front), and $s_1=1$. If $n\geq 2$, then
$s_n=3s_{n-1}-s_{n-2}$, since we get a segment of length $n$ if we affix a letter $B$, $C$, or $D$ at the end
of any segment of length $n-1$, except in the $s_{n-2}$ cases when this results in a $CB$-factor at the end. 
\end{proof}

We are now ready to announce our main tool of this section. 

Let $h_n$ be the number of the pairs of words $(w,z)$ over the alphabet $\{A,B,C,D\}$
that satisfy the following requirements. 
\begin{enumerate}
\item[(i)] Both $w$ and $z$ start with the letter $A$, and $|w|+|z|=n$. 
\item[(ii)] the words $w$ and $z$ contain the same number of letters $A$,
\item[(iii)] neither $w$ nor $z$ contains a $CB$-factor, and
\item[(iv)] for all $i$,  if the $i$th letter $A$ from the right in $w$ is in the middle of a $CAB$-factor,
then the $i$th segment of $z$ (from the left) contains a letter $B$. 
\end{enumerate}

Note that the words $w$ and $z$ do not have to have the same length.  Let $H(x)=\sum_{n\geq 2} h_n x^n$. 

\begin{theorem} \label{explicith}
The equality 
\[H(x)=\sum_{n\geq 2}h_nx^{n}= \frac{x^2(1-2x)}{x^6-5x^5+14x^4-26x^3+22x^2-8x+1}\] holds.
\end{theorem}

\begin{proof}
If a pair of words $(w,z)$ enumerated by $H(x)$ contains a total of two letters $A$, then both $w$ and $z$ must be a segment. Otherwise, removing the last segment of $w$ and the first segment of $z$, we get another, shorter pair
$(w',z')$ of words enumerated by $H(x)$. On the other hand, inserting a new last segment $S_1$ at the end of $w'$
and a new first segment $S_2$ at the front of $z'$, we get a new pair of words enumerated by $H(x)$ {\em except}
when the newly inserted last letter $A$ in $w$ is in the middle of a $CAB$-factor, {\em and} the new first
segment of $z$ does not contain any letters $B$. 

This leads to the generating function identity
\begin{equation} \label{funceq} 
H(x)=S^2(x)+S^2(x)H(x) - \frac{x}{1-2x} \cdot  \frac{x^2}{1-3x+x^2} \cdot xH(x).\end{equation}
The last summand of the right-hand side is justified as follows. The generating function for $S_2$ is $x^2/(1-2x)$ since $S_2$ starts with an $A$, and then consists of 
letters $C$ and $D$ with no restrictions, while the generating function for $S_1$ is $x^2/(1-3x+x^2)$, since this 
segment starts with $AB$, and then consists of a string over the alphabet $\{B,C,D\}$ with no $CB$-factors,
so it is just a segment with a $B$ inserted into its second position. 
Finally, the rest of the pair $(w,z)$ is just a word counted by $H(x)$ where a $C$ is inserted at the end of $w$.

Solving (\ref{funceq}) for $H(x)$, we get the statement of Theorem.  
\end{proof}

\begin{corollary} \label{words} There exists a positive constant $c$ so that the inequality 
\[h_n\leq c \cdot 3.709381^n \] holds.
\end{corollary}

\begin{proof} This is routine after noticing that the denominator of $ H(x)$ as given in Theorem \ref{explicith}
as a unique root of smallest modulus, namely $\alpha=0.2695867676$, and computing $\beta=1/\alpha =
3.709381$. 
\end{proof}

Now we can prove the main result of this section.

\begin{theorem} \label{main}
The inequality \[L(1324)\leq 13.7595074\] holds.
\end{theorem}

\begin{proof} This is immediate from the fact that if $p\in \hbox{Av}(1324)_n$, then the pair $w(p),z(p)$ is in the
set that is counted by $h_{2n}$ and is defined by rules (i)--(iv). Indeed, $w(p)$ and $z(p)$ are both words 
of length $n$ over   the alphabet $\{A,B,C,D\}$, they contain the same number of letters of each type, do not contain
any $CB$-factors as shown in Theorem \ref{nocb}, and the pair $(w(p),z(p))$ satisfy condition (iv) by Lemma \ref{firstconn}.
Therefore, $S_n(1324)\leq h_{2n}$, and our result follows from Corollary \ref{words} by computing $\beta^2$. 
\end{proof}

\section{Extending the reach of our method}

There are several other constraints that the pair $(w(p),z(p))$ must satisfy if $p$ is a 1324-avoiding permutation.
The problem is that it is difficult to count pairs that satisfy all these constraints at once. On the other hand, 
taking only a few constraints into account results only in a small improvement of the result of Theorem \ref{main}. 

For instance, Lemma \ref{firstconn} can be generalized in the following way. A $CAB^k$ factor in a word
is a subword of $k+2$ letters in consecutive positions, the first of which is a $C$, the second of which is an $A$, and 
all the others are $B$. 

\begin{lemma} \label{secondconn} Let $p$ be a 1324-avoiding permutation, and let $(w(p),z(p))$ be its image under
the coloring defined by rules (1)-(4) and (4'). 
Then for all $i$, and all $k$, the following holds.
If the $i$th letter $A$ from the right in $w(p)$ is in the middle of a $CAB^k$-factor, then the $i$th segment of
$z(p)$ from the left must contain at least $k$ letters $B$.
\end{lemma}

\begin{proof}
The proof is analogous to that of Lemma \ref{firstconn}. The only change is that the role of $b$ in that proof must
be played by the $k$th letter $B$ in the $CAB^k$-factor of the lemma. In order for $p$ to be 1324-avoiding, $b<a'$
must hold (keeping the notation of the proof of Lemma \ref{firstconn}). As $b$ is clearly the largest of the $k$ letters
$B$ immediately following $a$, this implies that there are at least $k$ letters $B$ between $a$ and $a'$ in $z(p)$. 
\end{proof}

Let us now check what numerical improvement we obtain if we enforce the constraint of Lemma \ref{secondconn} for
$k=2$. 
Let $k_n$ be the number of   pairs of words $(w,z)$ over the alphabet $\{A,B,C,D\}$ that satisfy all of the rules (i)-(iv),
 and also the following rule.

\begin{itemize}
\item[(v)]  For all $i$,  if the $i$th letter $A$ from the right in $w$ is in the middle of a $CABB$-factor,
then the $i$th segment of $z$ (from the left) contains at least two letters $B$. 
\end{itemize}
Let $K(x)=\sum_{n\geq 2}k_nx^n$. 

\begin{theorem} \label{newexplicit}
The identity 
\[K(x)=\frac{x^2(1-2x)^2}{1-10x+38x^2-70x^3+66x^4-33x^5+12x^6-6x^7+4x^8-x^9}\] holds.
\end{theorem}

\begin{proof}
If a pair of words $(w,z)$ enumerated by $K(x)$ contains a total of two letters $A$, then both $w$ and $z$ must be a segment. Otherwise, removing the last segment of $w$ and the first segment of $z$, we get another, shorter pair
$(w',z')$ of words enumerated by $K(x)$. On the other hand, inserting a new last segment $S_1$ at the end of $w'$
and a new first segment $S_2$ at the front of $z'$, we get a new pair of words enumerated by $H(x)$ {\em except}
in two disjoint cases. The first such case is when the newly inserted last letter $A$ in $w$ is in the middle of a
 $CAB$-factor, {\em and} the new first
segment of $z$ does not contain any letters $B$. The second such case is when the newly inserted last letter $A$ in
$w$ is in the middle of a $CABB$-factor, and the new first segment of $z$ contains {\em exactly one} letter $B$. 
This leads to the functional equation
\[K(x)=\frac{x^2}{(1-3x+x^2)^2} \cdot (1+K(x))  -  \frac{x}{1-2x} \cdot  \frac{x^2}{1-3x+x^2} \cdot xK(x)\]
\[-\frac{x^3}{1-3x+x^2}\cdot x\cdot\left(\frac{x}{1-2x}+1\right)x\cdot \frac{1}{1-2x} \cdot xK(x).\]
The first and second summand of the right-hand side can be explained just as in the proof of Theorem \ref{explicith}.
The third summand counts the pairs $(w,z)$ described as the second case in the previous paragraph. In such
pairs, $S_2$ is a segment containing exactly one letter $B$. Such segments are counted by the generating function
$ x\cdot\left(\frac{x}{1-2x}+1\right)x\cdot \frac{1}{1-2x} $, since in such segments, a letter $A$ is followed by a
(possibly empty) sequence of $C$s and $Ds$ that ends in a $D$ if it is not empty, then comes a letter $B$, and then
comes a sequence of $C$s and $Ds$ again. On the other hand, in such pairs, $S_1$ is just a segment into which
two letters $B$ are inserted right after the first letter, leading to the generating function $\frac{x^3}{1-3x+x^2}$.
Finally, the rest of the pair $(w,z)$ is just a pair counted by $K(x)$, with a letter $C$ inserted at the end of the
first word of the pair.  

Solving the last displayed equation for $K(x)$, we get the statement of the theorem. 
\end{proof}

\begin{corollary}
There exists a constant $c$ so that $k_n\leq c\cdot 3.70672^n$ holds for all $n$. 
\end{corollary}

\begin{proof} This follows from the fact that $K(x)$ is a rational function whose denominator has a unique root 
of smallest modulus, namely 0.26978. Taking its reciprocal, our claim is proved.
\end{proof}

\begin{theorem} \label{newbound}
The inequality \[L(1324)\leq 13.73977 \]
holds.
\end{theorem}

\begin{proof} This is immediate from the fact that $S_n(1324)\leq k_{2n}$. 
\end{proof}

It goes without saying that further improvements are possible if we enforce the restriction of Lemma \ref{secondconn}
not only for $k=2$, but for larger values of $k$ as well. However, these improvements will be minuscule. 
Indeed, let us replace restrictions (iv) and (v) by the following general restriction.

\begin{itemize}
\item[(vi)]    For all $i\geq 1$ and all $k\geq 1$,  if the $i$th letter $A$ from the right in $w$ is in the middle of a $CAB^k$-factor,
then the $i$th segment of $z$ (from the left) contains at least $k$ letters $B$.
\end{itemize}

Now let $t_n$ be the number of pairs $(w,z)$ that satisfy restrictions (i)-(iii) and (vi). Let $T(x)=\sum_{n\geq 2}t_nx^n$.
 
\begin{theorem} 
The identity \[T(x)=\frac{x^2(1-2x-x^2+x^3)}{1-8x+21x^2-19x^3-2x^4+11x^5-6x^6+x^7}\] holds.
\end{theorem}

\begin{proof} This follows from the functional equation
\[T(x)=\frac{x^2}{(1-3x+x^2)^2} \cdot (1+T(x)) -\frac{1}{1-2x}\cdot \frac{x^2}{1-3x+x^2} \cdot 
\sum_{j=0}\frac{x^2-x^3}{1-2x},\] which can be proved in a way that is analogous to the proofs of Theorems
\ref{explicith} and \ref{newexplicit}. 
\end{proof}   

 \begin{corollary} The inequality
\begin{equation} L(1324) \leq 13.73718 \end{equation} holds. 
\end{corollary}

\begin{proof} This is immediate if we notice that $S_n(1324)\leq t_{2n}$, and find the root of smallest modulus of the
denominator of $T(x)$. 
\end{proof} 

Other minor improvements can be obtained if one considers $CAAB$-factors in $w(p)$. 

A more substantial improvement could possibly be obtained in the following way. The restrictions involving $CAB$-factors
that we described have dual versions that involve $CDB$-factors. If one could find a way to count pairs of words that
satisfy {\em both kinds of restrictions} at once, (that is, restrictions involving $CDB$-factors and restrictions involving
$CAB$-factors), then a somewhat better upper bound could be obtained.

\end{document}